\documentclass{amsart}

\DeclareMathSymbol{\twoheadrightarrow}  {\mathrel}{AMSa}{"10}

                                 \def\uft{{\mathbf u}}
\def\Q{{\mathbb Q}}
\def\Z{{\mathbb Z}}
\def\C{{\mathbb C}}

                         \def\cc{{\mathfrak c}}

\def\RR{{\mathbb R}}
\def\F{{\mathbb F}}

                         \def\p{{\mathfrak p}}
                         \def\b{{\mathfrak b}}

                         \def\bb{{\mathbf b}}
                         \def\cc{{\mathbf c}}

\def\f{{\tilde F}}
                     \def\f0{{\mathfrak f}}

             \def\K{\mathrm{K}}

\def\A8{{\mathbf A}_8}

                                                      \def\Is{\mathrm{Is}}
                                                       \def\Isog{\mathrm{Isog}}
\def\RR{{\mathfrak R}}
\def\RR{{\mathfrak R}}
\def\Perm{\mathrm{Perm}}
\def\Gal{\mathrm{Gal}}

\def\End{\mathrm{End}}
\def\Aut{\mathrm{Aut}}

\def\Discr{\mathrm{Discr}}

\def\ST{{\mathbf S}}

\def\fchar{\mathrm{char}}

            \def\Gp{\mathrm{Gp}}

\def\A{\mathbf{A}}

\def\Oc{{\mathcal O}}

\newtheorem{thm}{Theorem}[section]
\newtheorem{lem}[thm]{Lemma}
\newtheorem{cor}[thm]{Corollary}

\theoremstyle{definition}

\newtheorem{ex}[thm]{Example}
\newtheorem{rem}[thm]{Remark}

\newtheorem{rems}[thm]{Remarks}
        \newtheorem{sect}[thm]{}
\title[Mori trinomials and monodromy]{Galois groups of Mori trinomials and  hyperelliptic curves with big monodromy}

\author[Yuri G.\ Zarhin]{Yuri G.\ Zarhin}
\thanks{This work was partially supported by a grant from the Simons Foundation (\#246625 to Yuri Zarkhin).}

\address{Department of Mathematics, Pennsylvania State University,
University Park, PA 16802, USA}

\email{zarhin\char`\@math.psu.edu}
\begin{document}

\begin{abstract}
We compute the Galois groups for a certain class of polynomials over
the the field of rational numbers  that was introduced by S. Mori
and study the monodromy of corresponding hyperelliptic jacobians.
\end{abstract}


\maketitle

\section{Mori polynomials, their reductions and Galois groups}

We write $\Z$, $\Q$ and $\C$  for the ring of integers, the field of
rational numbers
 and the field of complex numbers
respectively. If $a$ and $b$ are nonzero integers then we write
$(a,b)$ for its (positive) greatest common divisor. If $\ell$ is a
prime then  $\F_{\ell},\Z_{\ell}$ and $\Q_{\ell}$ stand for the
prime finite field of characteristic $\ell$, the ring of
${\ell}$-adic integers and the field of ${\ell}$-adic numbers
respectively.
 If $A$ and $B$ are nonzero integers then we write
$(A,B)$ for its greatest (positive) common divisor.

We consider the subring $\Z\left[\frac{1}{2}\right] \subset \Q$
generated by $1/2$ over $\Z$. We have
$$\Z\subset \Z\left[\frac{1}{2}\right] \subset \Q.$$
If $\ell$ is an odd prime then the principal ideal $\ell
\Z\left[\frac{1}{2}\right]$ is maximal  in
$\Z\left[\frac{1}{2}\right]$ and
$$\Z\left[\frac{1}{2}\right]/\ell\Z\left[\frac{1}{2}\right]=\Z/\ell\Z=\F_{\ell}.$$
If $K$ is a field then we write $\bar{K}$ for its algebraic closure
and denote by $\Gal(K)$ its absolute Galois group $\Aut(\bar{K}/K)$.
If $u(x)\in K[x]$ is a degree $n$ polynomial with coefficients in
$K$ and without  multiple roots then we write
 $\RR_u\subset \bar{K}$ for the $n$-element set of its roots, $K(\RR_u)$ the
 splitting field of $u(x)$ and $\Gal(u/K)=\Gal(K(\RR_u)/K)$  the
 Galois group of $u(x)$ viewed as a certain subgroup of the group
 $\Perm(\RR_u) \cong \ST_n$ of permutations of
 $\RR_u$. As usual, we write $\A_n$ for the {\sl alternating group}, which is the only index $2$ subgroup in the  {\sl full symmetric group} $\ST_n$.

\begin{sect}[{\bf Discriminants and alternating groups}]
\label{discrA}
 We write $\Delta(u)$ for the discriminant of $u$. We have
 $$0 \ne \Delta(u) \in K, \ \sqrt{\Delta(u)} \in K(\RR_u).$$
 It is well known that
$$\Gal(K(\RR_u)/K(\sqrt{\Delta(u)}))=\Gal(K(\RR_u)/K)\bigcap \A_n
\subset \A_n\subset \ST_n=\Perm(\RR_u).$$
In particular, the permutation (sub)group  $\Gal(K(\RR_u)/K(\sqrt{\Delta(u)}))$ does {\sl not} contain transpositions;
$\Delta(u)$ is a {\sl square} in $K$ if and only if
 $\Gal(u/K)$ lies in the {\sl alternating} (sub)group $\A_n\subset \ST_n$.
On the other hand, if $\Gal(u/K)=\ST_n$ then $\Gal(K(\RR_u)/K(\sqrt{\Delta(u)}))= \A_n$.
\end{sect}

  If $n$ is odd and $\fchar(K)\ne 2$ then we write $C_u$ for the genus
 $\frac{n-1}{2}$
 hyperelliptic curve
 $$C_u: y^2=u(x)$$
 and $J(C_u)$ for its jacobian, which is a
 $\frac{n-1}{2}$-dimensional abelian variety over $K$.
We write $\End(J(C_u))$ for the ring of all $\bar{K}$-endomorphisms
of  $J(C_u)$ and $\End_K(J(C_u))$ for the (sub)ring of all its
$K$-endomorphisms. We have
$$\Z\subset \End_K(J(C_u))\subset \End(J(C_u)).$$
About forty years ago S. Mori \cite[Prop. 3 on p. 107]{Mori}
observed that if $n=2g+1$ is odd and $\Gal(f/K)$ is a {\sl doubly
transitive} permutation group then $\End_K(J(C_u))=\Z$. He
constructed \cite[Th. 1 on p. 105]{Mori} explicit examples (in all
dimensions $g$) of polynomials (actually, trinomials) $f(x)$ over
$\Q$ such that $\Gal(f/\Q)$ is doubly transitive and
$\End(J(C_f))=\Z$.

On the other hand, about fifteen years ago the following assertion
was proven by the author \cite{ZarhinMRL}.

\begin{thm}
\label{endo} Suppose that $\fchar(K)=0$ and $\Gal(u/K)=\ST_n$. Then
$$\End(J(C_u))=\Z.$$
\end{thm}

The aim of this note is to prove that in Mori's examples
$\Gal(f/\Q)=\ST_{2g+1}$.  This gives another proof of the theorem of
Mori  \cite[Th. 1 on p. 105]{Mori}. Actually, we  extend  the class
of Mori trinomials  with  $\End(J(C_f))=\Z$, by dropping one of the
congruence conditions imposed by Mori on the coefficients of $f(x)$.
We also prove that the images of $\Gal(\Q)$ in the automorphism
groups of Tate modules of $J(C_f)$ are {\sl almost} as large as
possible.

\begin{sect}[{\bf Mori trinomials}]
\label{MoriP}
 Throughout this paper, $g,p,b,c$ {\sl are integers
that enjoy the following properties} \cite{Mori}.

\begin{itemize}
\item[(i)] {\sl The number $g$ is a positive integer and $p$ is an odd
prime. In addition, there is a positive integer $N$ such that
$(\frac{p-1}{2})^N$ is divisible by $g$}. This means that every {\sl
prime} divisor of $g$ is also a divisor of $\frac{p-1}{2}$. This
implies that
$$(p,g)=(p,2g)=1.$$
It follows that if $g$ is {\sl even} then $p$ is congruent to $1$
modulo $4$.

\item[(ii)] {\sl The residue $b\bmod p$ is a primitive root of
$\F_p=\Z/p\Z$}; in particular, $(b,p)=1$.

\item[(iii)]
 {\sl The integer $c$ is odd and}
$$(b,c)=(b,2g+1)=(c,g)=1.$$
This implies that $(c,2g)=1$.
\end{itemize}

S. Mori \cite{Mori} introduced and studied the  monic degree
$(2g+1)$ polynomial
$$f(x)=f_{g,p,b,c}(x):=x^{2g+1}-bx-\frac{pc}{4}\in \Z\left[\frac{1}{2}\right][x]\subset \Q[x],$$
 which we call a {\sl Mori trinomial}.  He proved the following
results \cite[pp. 106--107]{Mori}.
\end{sect}

\begin{thm}[Theorem of Mori]
\label{MoriT}
 Let $f(x)=f_{g,p,b,c}(x)$ be a Mori trinomial.
  Then:
\begin{enumerate}
\item[(i)]
The polynomial $f(x)$ is irreducible  over $\Q_2$ and therefore over
$\Q$.
\item[(ii)]
The polynomial $f(x)\bmod p \in \F_p[x]$ is a product $x (x^{2g}-b)$
of a linear factor $x$ and an irreducible (over $\F_p$) degree $2g$
polynomial $x^{2g}-b$.
\item[(iii)]
Let $\Gal(f)$ be the Galois group of $f(x)$ over $\Q$ considered
canonically as a (transitive) subgroup of  the full symmetric group
$\ST_{2g+1}$. Then $\Gal(f)$ is a doubly transitive permutation
group. More precisely, the transitive $\Gal(f)$ contains a
permutation $\sigma$ that is a cycle of length $2g$.
\item[(iv)]
For each odd prime $\ell$ every root of the polynomial $f(x)\bmod
\ell \in \F_{\ell}[x]$ is either simple or double.
\item[(v)]
 Let us
consider the genus $g$ hyperelliptic curve
$$C_f: y^2=f(x)$$
and its jacobian $J(C_f)$, which is a $g$-dimensional abelian
variety over $\Q$. Assume additionally that $c$ is congruent to $-p$
modulo $4$.

Then $C_f$ is a stable curve over $\Z$ and $J(C_f)$ has everywhere
semistable reduction over $\Z$. In addition, $\End(J(C_f)=\Z$.
\end{enumerate}
\end{thm}

\begin{rems}
\label{drop4}
\begin{itemize}
\item[(1)]
The $2$-adic Newton polygon of Mori trinomial $f(x)$ consists of one
{\sl segment} that connects $(0,-2)$ and $(2g+1,0)$, which  are its
only integer points. Now the irreducibility of $f(x)$ follows from
Eisenstein--Dumas Criterion \cite[Cor. 3.6 on p. 316]{Mott},
\cite[p. 502]{Gao}. It also follows that the field extension $\Q(\RR_f)/\Q$ is {\sl ramified} at $2$.
\item[(2)]
If $g=1$ then $2g+1=3$ and the only doubly transitive subgroup of
$\ST_{3}$ is $\ST_{3}$ itself. Concerning the double transitivity of
the Galois group of trinomials of arbitrary degree, see \cite[Th.
4.2 on p. 9 and Note 2 on p. 10]{Cohen}.
\item[(3)]
The additional congruence condition in Theorem \ref{MoriT}(v)
guarantees that $C_f$ has stable (even good) reduction at $2$
\cite[p. 106]{Mori}.
 Mori's proof of the last assertion of Theorem
\ref{MoriT}(v) is based on results of \cite{Ribet} and the equality
$\End_{\Q}(J(C_f))=\Z$; the latter follows from the double
transitivity of Galois groups of Mori trinomials.
\end{itemize}
\end{rems}

\begin{rem}
\label{An} Since a   cycle of {\sl even} length $2g$ is an {\sl
odd} permutation, it follows from Theorem \ref{MoriT}(iii) that
$\Gal(f)$ is {\sl not} contained in  $\A_{2g+1}$. In other words,
$\Delta(f)$ is {\sl not} a square in $\Q$.
\end{rem}

Our first main result is the following statement.

\begin{thm}
\label{main}
 Let $f(x)=f_{g,p,b,c}(x)$ be a Mori trinomial.
\begin{enumerate}
\item[(i)]
If $\ell$ is an odd prime then the polynomial $f(x)\bmod \ell \in
\F_{\ell}[x]$ has, at most, one double root and this root (if
exists) lies in $\F_{\ell}$.
\item[(ii)]
There exists an odd prime $\ell \ne p$ that $f(x)\bmod \ell \in
\F_{\ell}[x]$ has a double root $\bar{\alpha}\in \F_{\ell}$. All
other roots of $f(x)\bmod \ell$ (in an algebraic closure of
$\F_{\ell}$) are simple.
\item[(iii)]
The Galois group $\Gal(f)$  of $f(x)$ over $\Q$ coincides with the
full symmetric group $\ST_{2g+1}$. The Galois (sub)group $\Gal(\Q(\RR_f)/\Q(\sqrt{\Delta(f)}))$ coincides with the alternating group $\A_{2g+1}$.
\item[(iiibis)]
The Galois extension $\Q(\RR_f)/\Q(\sqrt{\Delta(f)})$ is ramified at all prime divisors of $2$. It is unramified at all prime divisors of every odd prime $\ell$.
\item[(iv)]
Suppose that $g>1$. Then  $\End(J(C_f))=\Z$.

\end{enumerate}
\end{thm}

\begin{rem}
  Theorem \ref{main}(iv) was proven by S. Mori
 under an additional assumption that $c$ is congruent to $-p$ modulo
$4$ (see Theorem \ref{MoriT}(v) above).
\end{rem}

\begin{rem}
\label{endoST} Thanks to Theorem \ref{endo}, Theorem \ref{main}(iv)
follows readily from Theorem \ref{main}(iii).
\end{rem}

\begin{rem}
\label{double} Let $g>1$ and suppose we know that $\Gal(f)$ contains
a transposition. Now the double transitivity implies that $\Gal(f)$
coincides with $\ST_{2g+1}$ (see \cite[Lemma 4.4.3 on p.
40]{SerreG}).
\end{rem}

Let $K$ be a field of characteristic zero and
 $u(x) \in K[x]$ be a degree $2g+1$ polynomial without multiple
roots. Then the jacobian $J(C_u)$ is a $g$-dimensional abelian
variety over $K$. For every prime $\ell$ let $T_{\ell}(J(C_u))$ be
the $\ell$-adic Tate module of $J(C_u)$, which is a free
$\Z_{\ell}$-module of rank $2g$ provided with the canonical
continuous action
$$\rho_{\ell,u}:\Gal(K) \to \Aut_{\Z_{\ell}}(T_{\ell}(J(C_u)))$$
of $\Gal(\Q)$ \cite{Mumford,SerreAbelian,ZarhinCEJM}. There is a {\sl Riemann
form}
$$e_{\ell}: T_{\ell}(J(C_u))\times T_{\ell}(J(C_u)) \to \Z_{\ell}$$
that corresponds to the canonical principal polarization on $J(C_u)$
(\cite[Sect. 20]{Mumford}, \cite[Sect. 1]{ZarhinTAMS}) and is a
nondegenerate (even perfect) alternating $\Z_{\ell}$-bilinear form
that satisfies
$$e_{\ell}(\sigma(x),\sigma(y))=\chi_{\ell}(\sigma)e_{\ell}(\sigma(x),\sigma(y)).$$
This implies that the image
$$\rho_{\ell,u}(\Gal(K))\subset \Aut_{\Z_{\ell}}(T_{\ell}(J(C_u)))$$
lies in the (sub)group
$$\Gp(T_{\ell}(J(C_u)),e_{\ell}) \subset \Aut_{\Z_{\ell}}(T_{\ell}(J(C_u)))$$
of symplectic similitudes of $e_{\ell}$
\cite{ZarhinMMJ,ZarhinPLMS,ZarhinTAMS}.

Using results of Chris Hall \cite{Hall} and the author
\cite{ZarhinTAMS}, we deduce from Theorem \ref{main} the following
statement. (Compare it with \cite[Th. 2.5]{ZarhinMMJ} and \cite[Th.
8.3]{ZarhinPLMS}.)

\begin{thm}
\label{main2}
 Let $K=\Q$ and $f(x)=f_{g,p,b,c}(x)\in \Q[x]$ be a Mori
trinomial. Suppose that $g>1$.

Then:

\begin{itemize}
\item[(i)]
 for all primes $\ell$ the image
$\rho_{\ell,f}(\Gal(\Q))$ is an open subgroup of finite index in
$\Gp(T_{\ell}(J(C_f)),e_{\ell})$.
\item[{ii}]
Let $L$ be a number field and $\Gal(L)$ be its absolute Galois
group, which we view as an open subgroup of finite index in
$\Gal(\Q)$. Then for all but finitely many primes $\ell$ the image
$\rho_{\ell,f}(\Gal(L))$ coincides with
$\Gp(T_{\ell}(J(C_f)),e_{\ell})$.
\end{itemize}
\end{thm}

The paper is organized as follows. In Section \ref{monodromy} we
deduce Theorem \ref{main2} from Theorem \ref{main}. In Section
\ref{reduction} we discuss a certain class of trinomials that is
related to Mori polynomials.  Section \ref{Dmori} deals with
discriminants of Mori polynomials.  We prove Theorem \ref{main} in
Section \ref{mainProof}.

{\bf Acknowledgements}. This work was started during my stay at the
Max-Planck-Institut f\"ur Mathematik (Bonn, Germany) in September of
2013 and finished
 during the academic year 2013/2014 when I was
 Erna and Jakob Michael Visiting Professor in the Department of Mathematics  at the Weizmann Institute of Science (Rehovot, Israel):
 the hospitality and support of both Institutes are gratefully acknowledged.

I am grateful to the referee, whose comments helped to improve the exposition.

\section{Monodromy of hyperelliptic jacobians}
\label{monodromy}

\begin{proof}[Proof of Theorem \ref{main2} (modulo Theorem
\ref{main})] By Theorem \ref{main}(iii), $\Gal(f/\Q)$ coincides with
the full symmetric group $\ST_{2g+1}$.  By Theorem \ref{main}(iv),
$\End(J(C_f))=\Z$.
 It follows
from Theorem \ref{main}(i) that there is an odd prime $\ell$ such
that $J(C_f)$ has at $\ell$ a semistable reduction with {\sl toric
dimension} $1$ \cite{Hall}.  Now the assertion (i) follows from
\cite[Th. 4.3]{ZarhinTAMS}. The assertion (ii) follows from
\cite[Th. 1]{Hall}.
\end{proof}

\section{Reduction of certain trinomials}
\label{reduction}
 In order to prove Theorem \ref{main}(i), we will use the
following elementary statement that was inspired by \cite[Remark 2
on p. 42]{SerreG} and \cite[p. 106]{Mori}

\begin{lem}[Key Lemma]
\label{onlyOne} Let
$$u(x)=u_{n,B,C}(x):=x^n+Bx+C \in \Z[x]$$
be a monic  polynomial of degree $n>1$ such that $B\ne 0$ and $C\ne
0$.

\begin{enumerate}
\item[(1)]
If $u(x)$ has a multiple root then  $n$ divides $B$ and $(n-1)$
divides $C$.

\item[(2)]
Let $\ell$ be a prime
 that enjoys the following properties.
\begin{itemize}
\item[(i)]
$(B,C)$ is not divisible by $\ell$.
\item[(ii)]
$(n,B)$  is not divisible by $\ell$.
\item[(iii)]
$(n-1,C)$ is not divisible by $\ell$.
\end{itemize}
Suppose that $u(x)$ has no multiple roots. Let us consider the
polynomial
$$\bar{u}(x):=u(x)\bmod \ell \in \F_{\ell}[x].$$
Then:
\begin{itemize}
\item[(a)]
 $\bar{u}(x)$ has, at most, one multiple root in an algebraic closure of $\F_{\ell}$.
\item[(b)]
 If such a multiple root say, $\gamma$, does exist, then $\ell$ does not divide
 \newline
 $n(n-1)BC$ and $\gamma$ is a double root of $\bar{u}(x)$. In addition, $\gamma$ is a nonzero element of
 $\F_{\ell}$.
 \item[(c)]
If such a multiple root  does exist then either the field extension
$\Q(\RR_u)/\Q$ is unramified at $\ell$ or a corresponding inertia
subgroup at $\ell$ in
$$\Gal(\Q(\RR_u)/\Q)=\Gal(u/\Q) \subset \Perm(\RR_u)$$
is generated by a transposition. In both cases the Galois extension $\Q(\RR_u)/\Q(\sqrt{\Delta(u)})$ is unramified at all prime divisors of $\ell$.
\end{itemize}
\end{enumerate}
\end{lem}

\begin{rem}
\label{discr} The discriminant
$$\Discr(n,B,C):=\Delta(u_{n,B,C})$$
 of
$u_{n,B,C}(x)$ is given by the formula
$$\Discr(n,B,C)=(-1)^{n(n-1)/2} n^n C^{n-1}+ (-1)^{(n-1)(n-2)/2} (n-1)^{n-1} B^n$$
\cite[Ex. 834]{FS}.
\end{rem}

\begin{rem}
\label{trivial}
In the notation  of Lemma \ref{onlyOne}, assume that $\bar{u}(x)$ has {\sl no} multiple roots, i.e., $\Delta(u)$ is {\sl not} divisible by $\ell$. Then obviously  $\Q(\RR_u)/\Q$ is {\sl unramified} at $\ell$. This implies that $\Q(\RR_u)/\Q(\sqrt{\Delta(u)})$ is {\sl unramified} at all prime divisors of $\ell$.
\end{rem}

\begin{proof}[Proof of Lemma \ref{onlyOne}]

{\sl Proof of} (1). Since $u(x)$ has a multiple root, its
discriminant
$$\Delta(u)=(-1)^{n(n-1)/2} n^n C^{n-1}+ (-1)^{(n-1)(n-2)/2} (n-1)^{n-1}
B^n=0.$$ This implies that
$$n^n C^{n-1}= \pm (n-1)^{n-1}
B^n.$$ Since $n$ and $(n-1)$ are relatively prime, $n^n\mid B^n$
and $(n-1)^{n-1} \mid C^{n-1}$. This implies that $n\mid B$ and $
(n-1)\mid C$.

{\sl Proof of} (2).
 We have
$$\bar{u}(x):=x^n+\bar{B}x+\bar{C} \in \F_{\ell}[x]$$ where
$$\bar{B}=B\bmod \ell \in \F_{\ell}, \ \bar{C}=C\bmod \ell \in
\F_{\ell}.$$ The condition (i) implies that either
 $\bar{B}\ne 0$ or $\bar{C} \ne 0$.
 The condition (ii) implies that
 if $\bar{B}= 0$ then $n \ne 0$ in $\F_{\ell}$.
 The condition (iii)
 implies that if $(n-1)=0$ in $\F_{\ell}$ then $\bar{C} \ne 0$
and  $n \ne 0$ in $\F_{\ell}$.
We have
$$\Delta(\bar{u})=(-1)^{n(n-1)/2} n^n \bar{C}^{n-1}+ (-1)^{(n-1)(n-2)/2} (n-1)^{n-1}
\bar{B}^n=0$$
and therefore
\begin{equation}\label{f1}
n^n \bar{C}^{n-1}= \pm (n-1)^{n-1}
\bar{B}^n.
\end{equation}
This implies that if $(n-1)=0$ in $\F_{\ell}$  then $ \bar{C}=0$, which is not the case. This proves that $(n-1)\ne 0$ in $\F_{\ell}$. On the other hand, if $\bar{B}= 0$ then $\bar{C} \ne 0$ and $n \ne 0$ in $\F_{\ell}$.  Then  (1) implies that $\bar{C} = 0$ and we get a contradiction that proves that $\bar{B}\ne 0$.  If $n=0$ in $\F_{\ell}$ then $n-1 \ne 0$ in $\F_{\ell}$ and  formula \ref{f1}  implies that  $\bar{B}= 0$, which is not the case. The obtained contradiction proves that $n \ne 0$ in $\F_{\ell}$. If $\bar{C} = 0$ then   formula (\ref{f1})  implies that  $\bar{B}= 0$, which is not the case. This proves that $\ell$ does {\sl not} divide
 $n(n-1)BC$.

The derivative of $\bar{u}(x)$ is
$$\bar{u}^{\prime}(x)=n x^{n-1}+\bar{B}.$$
We have
\begin{equation}\label{f2}
x \cdot \bar{u}^{\prime}(x)-n\cdot \bar{u}(x)=-(n-1)\bar{B}x-n\bar{C}.
\end{equation}
 Suppose $\bar{u}(x)$ has a multiple root  $\gamma$ in an
algebraic closure of $\F_{\ell}$. Then
$$\bar{u}(\gamma)=0, \ \bar{u}^{\prime}(\gamma)=0, \
n \cdot\gamma\cdot \bar{u}^{\prime}(\gamma)-n\cdot\bar{u}(\gamma)=0.$$
Using formula (\ref{f2}), we conclude that
$$0=\gamma \cdot\bar{u}^{\prime}(\gamma)-n\cdot\bar{u}(\gamma)= -(n-1)\bar{B}\gamma -n\bar{C}, \
\gamma=-\frac{n\bar{C}}{(n-1)\bar{B}}\in \F_{\ell}. $$
This implies that $\gamma \ne 0$.

Notice that the second derivative
$\bar{u}^{\prime\prime}(x)=n(n-1)x^{n-2}$. This implies that
$$\bar{u}^{\prime\prime}(\gamma)=n(n-1)\gamma^{n-2} \ne 0.$$
It follows that $\gamma$ is a {\sl double} root of $\bar{u}(x)$.
This ends the proof of (a) and (b).

In order to prove (c), notice that there exists a monic degree
$(n-2)$ polynomial $\bar{h}(x)\in \F_{\ell}[x]$ such that
$$\bar{u}(x)=(x-\gamma)^2 \cdot \bar{h}(x).$$
Clearly, $\gamma$ is {\sl not} a root of $\bar{h}(x)$ and therefore
$\bar{h}(x)$ has no multiple roots and relatively prime to
$(x-\gamma)^2$.\footnote{Compare with \cite[Lemma 1 on p. 231]{Osada}.}
 By Hensel's Lemma, there exist monic  polynomials
$$h(x), v(x) \in \Z_{\ell}[x], \ \deg(h)=n-2, \deg(v)=2$$ such that
$$u(x)=v(x) h(x)$$ and
$$\bar{h}(x)=h(x)\bmod \ell, \ (x-\gamma)^2=v(x) \bmod \ell.$$
This implies that the splitting field $\Q_{\ell}(\RR_h)$ of $h(x)$
(over $\Q_{\ell}$) is an unramified extension of $\Q_{\ell}$ while
the splitting field $\Q_{\ell}(\RR_u)$ of $u(x)$ (over $\Q_{\ell}$)
is obtained from $\Q_{\ell}(\RR_h)$ by adjoining to it two
(distinct) roots say, $\alpha_1$ and $\alpha_2$ of quadratic $v(x)$.
Clearly,  $\Q_{\ell}(\RR_u)$  either coincides with
$\Q_{\ell}(\RR_h)$ or with a certain quadratic extension of
$\Q_{\ell}(\RR_h)$, ramified or unramified. It follows that the
inertia subgroup $I$ of
$$\Gal(\Q_{\ell}(\RR_u)/\Q_{\ell})\subset \Perm(\RR_u)$$ is either
trivial or is generated by the {\sl transposition} that permutes
$\alpha_1$ and $\alpha_2$ (and leaves invariant every  root of
$h(x)$). In the former  case $\Q(\RR_u)/\Q$ is unramified at $\ell$
while in the latter one an inertia subgroup in
$$\Gal(\Q(\RR_u)/\Q)\subset \Perm(\RR_u)$$
that corresponds to $\ell$ is generated by a transposition. However, the permutation subgroup $\Gal(\Q(\RR_u)/\Q(\sqrt{\Delta(u)}))$ does not contain transpositions (see Sect. \ref{discrA}). This implies that $\Q(\RR_u)/\Q(\sqrt{\Delta(u)})$ is {\sl unramified} at all prime divisors of $\ell$.
\end{proof}

\begin{ex}
\label{xnx1} Let us consider the polynomial
$$u(x)=u_{n,-1,-1}(x)=x^n-x-1\in \Q[x]$$
over the field $K=\Q$.
Here $B=C=-1$ and the
conditions of Lemma \ref{onlyOne} hold for all primes $\ell$. It is
known  that $u(x)$ is irreducible \cite{Selmer}, its Galois group over $\Q$ is $\ST_n$ \cite[Cor. 3 on p. 233]{Osada} and there exists
a prime $\ell$ such that $u(x)\bmod \ell$ acquires a multiple root \cite[Remark 2 on p. 42]{SerreG}.
Clearly, the discriminant $\Delta(u)=\Discr(n,-1,-1)$ of $u(x)$ is
an {\sl odd} integer and therefore such an $\ell$ is {\sl odd}. It
follows from Lemma \ref{onlyOne} that $u(x)\bmod \ell$ has exactly
one multiple root and its multiplicity is $2$.

Let  $n=2g+1$ be an odd integer $\ge 5$ and
$$u(x)=u_{2g+1,-1,-1}(x)=x^{2g+1}-x-1\in \Q[x].$$ Let us consider the $g$-dimensional jacobian
$J(C_u)$ of the hyperelliptic curve $C_u:y^2=x^{2g+1}-x-1$. Since
$\Gal(u/\Q)=\ST_{2g+1}$, Theorem \ref{endo} tells us that
$\End(J(C_u))=\Z$. Now the same arguments as in Section
\ref{monodromy} prove that:

\begin{itemize}
\item[(i)]
 For all primes $\ell$ the image
$$\rho_{\ell,u}(\Gal(\Q))\subset \Gp(T_{\ell}(J(C_u)),e_{\ell})$$
is an open subgroup of finite index in
$\Gp(T_{\ell}(J(C_u)),e_{\ell})$.
\item[(ii)]
Let $L$ be a number field and $\Gal(L)$ be its absolute Galois
group, which we view as an open subgroup of finite index in
$\Gal(\Q)$. Then for all but finitely many primes $\ell$ the image
$$\rho_{\ell,u}(\Gal(L))\subset \Gp(T_{\ell}(J(C_u)),e_{\ell})$$
coincides with $\Gp(T_{\ell}(J(C_u)),e_{\ell})$.
\end{itemize}
\end{ex}

\begin{cor}[Corollary to Lemma \ref{onlyOne}]
\label{keycor} Let
$$u(x)=u_{n,B,C}(x):=x^n+Bx+C \in \Z[x]$$
be a monic  polynomial of degree $n>1$ without multiple roots such
that both $B$ and $C$ are nonzero integers that enjoy the following
properties.

\begin{itemize}
\item[(i)]
$(B,C)$ is either $1$ or a power of $2$.
\item[(ii)]
$(n,B)$ is either $1$ or a power of $2$.
\item[(iii)]
$(n-1,C)$ is either $1$ or a power of $2$.
\end{itemize}

Suppose that the discriminant $D=\Discr(n,B,C)= 2^{2M}\cdot D_0$
where $M$ is a nonnegative integer and $D_0$ is an integer
such that
$$D_0\equiv 1 \bmod 4.$$
Assume also that $D$ is not a square.
 Then:

\begin{itemize}
\item[(a)]
The quadratic extension $\Q(\sqrt{D})/\Q$ is unramified ar $2$. For all odd primes $\ell$ the Galois extension
$\Q(\RR_u)/\Q(\sqrt{D})$ is unramified at every prime divisor of $\ell$.

\item[(b)]
 There exists an odd
prime $\ell$ that enjoys the following properties.
\begin{itemize}
\item[(i)]
$\ell$ divides $D_0$ and
$$u(x)\bmod \ell \in \F_{\ell}[x]$$
has exactly one multiple root and its multiplicity is $2$. In
addition, this root lies in $\F_{\ell}$.
\item[(ii)]
The field extension $\Q(\RR_u)/\Q$ is ramified at $\ell$ and the
Galois group
$$\Gal(\Q(\RR_u)/\Q)=\Gal(u/\Q)\subset \Perm(\RR_u)$$
contains a transposition. In particular, if $\Gal(u/\Q)$ is doubly
transitive then
$$\Gal(u/\Q)=\Perm(\RR_f)\cong \ST_n$$
and
$$\Gal(\Q(\RR_u)/\Q(\sqrt{D}))=\A_n.$$
\end{itemize}
\end{itemize}
\end{cor}

\begin{proof}
Clearly, $D_0$ is {\sl not} a square and
$$\Q(\sqrt{D})=\Q(\sqrt{D_0})$$
is a quadratic field.
Since $D_0$ is congruent to $1$ modulo $4$, the quadratic extension $\Q(\sqrt{D_0})/\Q$ is {\sl unramified} at $2$, which proves  the first assertion of (a).
The conditions of Lemma \ref{onlyOne}(2) hold for all odd primes
$\ell$. Now the second assertion of (a) follows from Remark \ref{trivial} and Lemma \ref{onlyOne}(2)(c).

Let us start to prove (b).
 There are nonzero integers $S$ and
$S_0$ such that $D_0=S^2 S_0$ and $S_0$ is square-free. Clearly,
both $S$ and $S_0$ are odd. Since
$$D=2^{2M}\cdot D_0=2^{2M}\cdot S^2 S_0=\left(2^M S\right)^2 S_0$$
is {\sl not} a square, $S_0 \ne 1$. Since $S$ is odd,  $S^2\equiv 1 \bmod
4$. Since $D_0\equiv 1 \bmod 4$, we obtain that $S_0\equiv 1 \bmod
4$. It follows that $S_0 \ne -1$. We already know that $S_0 \ne 1$.
This implies that there is a prime $\ell$ that divides $S_0$. Since
$S_0$ is odd and square-free, $\ell$ is also odd and $\ell^2$ does
not divide $S_0$. Let $T$ be the nonnegative integer such that
$\ell^T\mid\mid S$. Then $\ell^{2T+1}\mid\mid 2^{2M}S^2 S_0$, and
therefore $\ell^{2T+1}\mid\mid D$. This implies that the quadratic
field extension $\Q(\sqrt{D})/\Q$ is {\sl ramified} at $\ell$. Since
$$\Q\subset \Q(\sqrt{D})\subset \Q(\RR_u),$$
the field extension $\Q(\RR_u)/\Q$ is also ramified at $\ell$.
 Since $\ell\mid D$, the polynomial
 $$u(x) \bmod \ell \in \F_{\ell}[x]$$
 has a multiple root.
 Now the result follows from Lemma \ref{onlyOne} combined with Remark \ref{double}.
\end{proof}

\section{Discriminants of Mori trinomials}
\label{Dmori}
 Let $$f(x)=f_{g,p,b,c}(x)=x^{2g+1}-bx-\frac{pc}{4}$$
be a Mori trinomial. Following Mori \cite{Mori}, let us consider the
polynomial
$$\uft(x)=2^{2g+1} f(x/2)=x^{2g+1}-2^{2g}b x-2^{2g-1}pc=u_{n,B,C}(x) \in \Z[x]\subset
\Q[x]$$ with
$$n=2g+1, B=-2^{2g}b, C= -2^{2g-1}pc.$$

\begin{rems}
\label{Doddg}
\begin{itemize}
\item[(i)]
Clearly, $f(x)$ and $\uft(x)$ have the same splitting field and
Galois group. It is also clear that
$$\Delta(\uft)=2^{(2g+1)2g}\cdot \Delta(f)=\left[2^{(2g+1)g}\right]^2 \cdot \Delta(f).$$
In particular, $\Delta(\uft)$ is {\sl not} a square, thanks to
Remark \ref{An}.

\item[(ii)]
By Theorem \ref{MoriT}(i,iii), the polynomial $f(x)$ is irreducible
over $\Q$ and its Galois group is {\sl doubly transitive}. This
implies that $\uft(x)$ is irreducible over $\Q$ and its Galois group
over $\Q$ is also {\sl doubly transitive}. (See also Theorem \ref{MoriG}(i,ii) below.)

\item[(iii)]
For all $g$ the hyperelliptic curves $C_f$ and $C_{\uft}$ are
biregularly isomorphic  over $\Q(\sqrt{2})$. It follows that the
jacobians $J(C_{\uft})$ and $J(C_f)$ are also isomorphic  over
$\Q(\sqrt{2})$. In particular, $\End(J(C_{\uft}))=\End(J(C_f))$.
\end{itemize}
\end{rems}

Clearly, the conditions of Lemma \ref{onlyOne} hold for
$u(x)=\uft(x)$ for all odd primes $\ell$. The discriminant
$\Delta(\uft)$ of $\uft(x)$ coincides with
$$\Discr(n,B,C)=(-1)^{(2g+1)2g/2} (2g+1)^{2g+1} [-2^{2g-1}pc]^{2g}+ (-1)^{2g(2g-1)/2} (2g)^{2g} [-2^{2g}b]^{2g+1}.$$
It follows that
$$\Delta(\uft)=(-1)^{g}2^{2g(2g-1)}\left[(2g+1)^{2g+1}(pc)^{2g}-2^{6g}g^{2g}b^{2g+1}\right].$$
This implies that

\begin{equation}
\label{f3}
\Delta(\uft) =2^{2[g(2g-1)]} D_0
\end{equation}
where
$$D_0=(-1)^{g}\left\{(2g+1) \left[(2g+1)^{g} (pc)^{g}\right]^2-2^{6g}  g^{2g} b^{2g+1}\right\}.$$
Clearly, $D_0$ is an {\sl odd} integer that is {\sl not} divisible
by $p$. It is also clear that
 $D_0$ is congruent to $(-1)^{g}(2g+1)$ modulo $4$
(because every odd square is congruent to $1$ modulo $4$). This
implies that
\begin{equation}
\label{f4} D_0 \equiv 1 \bmod 4
\end{equation}
for all $g$.

\section{Proof of Theorem \ref{main}}
\label{mainProof} Let us apply Lemma \ref{onlyOne}(ii) to
$$\uft(x)=2^{2g+1} f(x/2)=x^{2g+1}-2^{2g}bx-2^{2g-1}pc.$$
We obtain that for each odd prime $\ell$ the polynomial
$$\uft(x)\bmod \ell \in \F_{\ell}[x]$$
has, at most, one multiple root; in addition, this root is double
and lies in $\F_{\ell}$.  Applying to $\uft(x)$ Corollary
\ref{keycor} combined with  formulas (\ref{f3}) and (\ref{f4}) of Sect.
\ref{Dmori}, we conclude that there exists an odd $\ell \ne p$ such
that $\uft(x)\bmod \ell$ has exactly one multiple root; this root is
double and  lies in $\F_{\ell}$. In addition, $\Gal(\uft/\Q)$
coincides with $\ST_{2g+1}$, because it is doubly transitive. Now
the assertions (i) and (ii) follow readily from the equality
$$f(x)\bmod \ell =\frac{\uft(2x)}{2^{2g+1}} \bmod \ell.$$
that holds for all odd primes $\ell$.

 By Remarks \ref{Doddg},
$\Gal(f/\Q)=\Gal(\uft/\Q)$ and therefore also coincides with
$\ST_{2g+1}$, which implies (in light of Section \ref{discrA}) that
$\Gal(\Q(\RR_f)/\Q(\sqrt{\Delta(f)}))=\A_{2g+1}$. This proves (iii).
Now Remark \ref{endoST} implies that $\End(J(C_f))=\Z$. This proves
(iv). In order to prove (iiibis), first notice that the Galois
extension $\Q(\RR_f)/\Q$ is ramified at $2$ (Remark \ref{drop4}(1) )
while $\Q(\sqrt{\Delta(f)})=\Q(\sqrt{\Delta(u)})$ is unramified at
$2$ over $\Q$ in light of formulas (3) and (4) in Sect. \ref{Dmori}
(and Corollary \ref{keycor}(a)). This implies that
$\Q(\RR_f)/\Q(\sqrt{\Delta(f)})$ is ramified at some prime divisor
of $2$. Since all the field extensions involved are Galois,
$\Q(\RR_f)/\Q(\sqrt{\Delta(f)})$ is actually ramified at {\sl all}
prime divisors of $2$.   This proves the first assertion of
(iiibis). The second assertion of (iiibis) follows from Corollary
\ref{keycor}(a). This proves (iiibis).

\section{Variants and Complements}

Throughout this section,  $K$ is a number field. We write $\Oc$ for
the ring of integers in $K$. If $\b$ is a maximal ideal in $\Oc$
then we write $k(\b)$ for the (finite) residue field $\Oc/\b$. As usual, we
call $\fchar(k(\b))$ the residual characteristic of $\b$. We write
$K_{\b}$ for the $\b$-adic completion  of $K$ and
$$\Oc_{\b}\subset K_{\b}$$
for the ring of $\b$-adic integers in the field $K_{\b}$. We
consider the subring $\Oc\left[\frac{1}{2}\right]\subset K$
generated by $\frac{1}{2}$ over $\Oc$. We have
$$\Oc \subset \Oc\left[\frac{1}{2}\right]\subset K.$$
If $\b \subset \Oc$ is a maximal ideal in $\Oc$ with odd residual
characteristic then
$$\Oc \subset \Oc\left[\frac{1}{2}\right]\subset \Oc_{\b},$$
the ideal $\b \Oc\left[\frac{1}{2}\right]$ is a maximal ideal in
$\Oc\left[\frac{1}{2}\right]$ and
$$k(\b)=\Oc/\b=\Oc\left[\frac{1}{2}\right]/\b
\Oc\left[\frac{1}{2}\right]=\Oc_{\b}/\b \Oc_{\b}.$$

Lemma \ref{onlyOne}(ii) and its proof admit the following
straightforward generalization.

\begin{lem}
\label{onlyOneii}
 Let
$$u(x)=u_{n,B,C}(x):=x^n+Bx+C \in \Oc[x]$$
be a monic  polynomial of degree $n>1$ such that  $B\ne 0$ and $C\ne
0$.  Let $\b$ be a maximal ideal in $\Oc$  that
 enjoys the following properties.
\begin{itemize}
\item[(i)]
$B \Oc+C\Oc + \b=\Oc$.
\item[(ii)]
$n\Oc+B\Oc + \b=\Oc$.
\item[(iii)]
$(n-1)\Oc+C\Oc + \b=\Oc$.
\end{itemize}

Suppose that $u(x)$ has no multiple roots. Let us consider the
polynomial
$$\bar{u}(x):=u(x)\bmod \b \in k(\b)[x].$$
Then:
\begin{itemize}
\item[(a)]
 $\bar{u}(x)$ has, at most, one multiple root in an algebraic closure of $k(\b)$.
\item[(b)]
 If such a multiple root say, $\gamma$, does exist, then
$$n(n-1)BC \not\in \b$$
 and $\gamma$ is a double root of $\bar{u}(x)$. In
addition, $\gamma$ is a nonzero element of
 $k(\b)$.
 \item[(c)]
If such a multiple root  does exist then either the field extension
$K(\RR_u)/K$ is unramified at $\b$ or a corresponding inertia
subgroup at $\b$ in
$$\Gal(K(\RR_u)/K)=\Gal(u/\K) \subset \Perm(\RR_u)$$
is generated by a transposition.  In both cases the Galois extension  $K(\RR_u)/K(\sqrt{\Delta(u)})$ is unramified at all prime divisors of $\b$.
\end{itemize}
\end{lem}

\begin{rem}
\label{trivialK}
In the notation of Lemma \ref{onlyOneii}, suppose that $\bar{u}(x)$ has {\sl no} multiple roots, i.e., $\Delta(u) \not \in \b$. Then clearly the Galois extension $K(\RR_u)/K$ is unramified at $\b$.
\end{rem}

\begin{proof}
We have
$$\bar{u}(x):=x^n+\bar{B}x+\bar{C} \in k(\b)[x]$$ where
$$\bar{B}=B\bmod \b \in k(\b), \ \bar{C}=C\bmod \b \in
k(\b).$$
The condition (i) implies that either
 $\bar{B}\ne 0$ or $\bar{C} \ne 0$.
 The condition (ii) implies that
 if $\bar{B}= 0$ then $n \ne 0$ in $k(\b)$.
 It follows that if $\bar{B}= 0$ then $n\bar{C} \ne 0$.

 The condition (iii)
 implies that if $(n-1)=0$ in $k(\b)$ then $\bar{C} \ne 0$ (and, of course, $n \ne 0$ in
 $k(\b)$). On the other hand, if $\bar{C}=0$ then $(n-1) \ne 0$ in
 $k(\b)$.

 Suppose $\bar{u}(x)$ has a multiple root  $\gamma$ in an
algebraic closure of $k(\b)$. Then as in the proof of Lemma \ref{onlyOne}(2),
$$0=\Delta(\bar{u})=(-1)^{n(n-1)/2} n^n \bar{C}^{n-1}+ (-1)^{(n-1)(n-2)/2} (n-1)^{n-1}
\bar{B}^n=0.$$
This implies that
\begin{equation}
\label{f5}
n^n \bar{C}^{n-1}= \pm (n-1)^{n-1}
\bar{B}^n.
\end{equation}
This implies that if $(n-1)=0$ in $k(\b)$  then $ \bar{C}=0$, which is not the case. This proves that $(n-1)\ne 0$ in $k(\b)$. On the other hand, if $\bar{B}= 0$ then $\bar{C} \ne 0$ and $n \ne 0$ in $k(\b)$.  Then  formula (\ref{f5}) implies that $\bar{C} = 0$ and we get a contradiction that proves that $\bar{B}\ne 0$.  If $n=0$ in $k(\b)$ then $n-1 \ne 0$ in $k(\b)$ and    formula (\ref{f5})  implies that  $\bar{B}= 0$, which is not the case. The obtained contradiction proves that $n \ne 0$ in $k(\b)$. If $\bar{C} = 0$ then   formula (\ref{f5})  implies that  $\bar{B}= 0$, which is not the case. This proves that the maximal ideal $\b$ does {\sl not} contain
 $n(n-1)BC$.

On the other hand, we have
  as in the proof of Lemma \ref{onlyOne}(2) that
$$x \cdot\bar{u}^{\prime}(x)-n\cdot\bar{u}(x)=-(n-1)\bar{B}x-n\bar{C}$$
and therefore
$$-(n-1)\bar{B}\gamma-n\bar{C}=0.$$
 It follows that
$$\gamma=-\frac{n\bar{C}}{(n-1)\bar{B}}$$
is a {\sl nonzero} element of $k(\b)$.  The second derivative $\bar{u}^{\prime\prime}(x)=n(n-1)x^{n-2}$
 and
$$\bar{u}^{\prime\prime}(\gamma)=n(n-1)\gamma^{n-2} \ne 0.$$
It follows that $\gamma$ is a {\sl double} root of $\bar{u}(x)$.
This proves (a) and (b).

In order to prove (c),  notice that as in the proof of Lemma
\ref{onlyOne}(ii)(c), there exists a monic degree $(n-2)$ polynomial
$\bar{h}(x)\in k(\b)[x]$ such that
$$\bar{u}(x)=(x-\gamma)^2 \cdot \bar{h}(x)$$
and $\bar{h}(x)$ and $(x-\gamma)^2$ are relatively prime.  By
Hensel's Lemma, there exist monic  polynomials
$$h(x), v(x) \in \Oc_{\b}[x], \ \deg(h)=n-2, \deg(v)=2$$ such that
$$u(x)=v(x) h(x)$$ and
$$\bar{h}(x)=h(x)\bmod \b, \ (x-\gamma)^2=v(x) \bmod \b.$$
This implies that the splitting field $K_{\b}(\RR_h)$ of $h(x)$
(over $K_{\b}$) is an unramified extension of $\K_{\b}$ while the
splitting field $\K_{\b}(\RR_u)$ of $u(x)$ (over $\K_{\b}$) is
obtained from $\K_{\b}(\RR_h)$ by adjoining to it two (distinct)
roots say, $\alpha_1$ and $\alpha_2$ of quadratic $v(x)$. The field
$K_{\b}(\RR_u)$   coincides either with $K_{\b}(\RR_h)$ or with a
certain quadratic extension of $\K_{\b}(\RR_h)$, ramified or
unramified. It follows that the inertia subgroup $I$ of
$$\Gal(K_{\b}(\RR_u)/K_{\b})\subset \Perm(\RR_u)$$ is either
trivial or is generated by the {\sl transposition} that permutes
$\alpha_1$ and $\alpha_2$ (and leaves invariant every  root of
$h(x)$). In the former  case $K(\RR_u)/K$ is unramified at $\b$
while in the latter one an inertia subgroup in
$$\Gal(K(\RR_u)/K)\subset \Perm(\RR_u)$$
that corresponds to $\b$ is generated by a transposition. In both cases the Galois (sub)group $\Gal(K(\RR_u)/K(\sqrt{\Delta(u)}))$ does not contain transpositions (see Sect. \ref{discrA}). This implies that $K(\RR_u)/K(\sqrt{\Delta(u)})$ is {\sl unramified} at all prime divisors of $\b$.
\end{proof}

Corollary \ref{keycor} admits the following partial generalization.

\begin{lem}
\label{ODDclass} Let $K$ be a number field and $\Oc$ be its ring of
integers
 Let
$$u(x)=u_{n,B,C}(x):=x^n+Bx+C \in \Oc[x]$$
be a monic  polynomial without multiple roots of degree $n>1$ such
that both $B$ and $C$ are not zeros.  Suppose that there is a
nonnegative integer $N$ such that
$$2^N \Oc \subset B \Oc+C\Oc,  \ 2^N \Oc \subset n\Oc+B\Oc, \ 2^N \Oc\subset  (n-1)\Oc+C\Oc.$$
Suppose that  there is a nonnegative integer $M$ such that the
discriminant $D:=\Delta(u)=2^{2M}\cdot D_0$ with $D_0 \in \Oc$.

Assume also that $D, D_0$ and $K$ enjoy the following properties.
\begin{itemize}
\item[(i)]
$D$ is not a square in $K$ and
$$D_0-1 \in 4\Oc.$$
\item[(ii)]
The class number of $K$ is odd (e.g., $\Oc$ is a principal ideal
domain).
\item[(iii)]
Either $K$ is totally imaginary, i.e., it does not admit an
embedding into the field of real numbers or $K$ is totally real and
$D_0$ is totally positive.
\end{itemize}

Then:
\begin{itemize}
\item[(a)]
The quadratic extension $K(\sqrt{\Delta(u)})/K$ is unramified at every prime divisor of $2$. The Galois extension $K(\RR_u)/K(\sqrt{\Delta(u)})$ is unramified at every prime ideal $\b$ of odd residual characteristic.

\item[(b)]
There exists a maximal ideal $\b \subset \Oc$ with residue field $k(\b)$ of odd characteristic  that enjoys the following properties.

\begin{itemize}
\item[(i)]
$D_0 \in \b$, the polynomial
$$u(x)\bmod \b \in k(\b)[x]$$
has exactly one multiple root and its multiplicity is $2$. In
addition, this root lies in $k(\b)$.
\item[(ii)]
The field extension $K(\RR_u)/K$ is ramified at $\b$ and the Galois
group
$$\Gal(K(\RR_u)/K)=\Gal(u/\K)\subset \Perm(\RR_u)$$
contains a transposition. In particular, if $\Gal(u/\K)$ is doubly
transitive then $$\Gal(u/\K)=\Perm(\RR_f)\cong \ST_n$$
and
$$\Gal(K(\RR_u)/K(\sqrt{\Delta(u)}))=\A_n.$$
\end{itemize}
\end{itemize}
\end{lem}

\begin{proof}
Let us prove (a).
Clearly,
$$E:=K\left(\sqrt{D_0}\right)=K\left(\sqrt{D}\right)=K\left(\sqrt{\Delta(u)}\right)\subset K(\RR_u)$$
is a quadratic extension of $K$. Notice that
$$\theta=\frac{1+\sqrt{D_0}}{2} \in E$$
is a root of the quadratic equation
$$v_2(x):=x^2-x +\frac{1-D_0}{4} \in \Oc[x]$$
and therefore is an algebraic integer. In addition,
$$E=K(\theta).$$
If a maximal ideal $\b_2$ in $\Oc$ has residual characteristic $2$
then the quadratic polynomial
$$v_2(x)\bmod \b_2 =x^2-x +\left(\frac{1-D_0}{4}\right)\bmod\b_2 \in k(\b_2)[x]$$
has no multiple roots, because its derivative is a nonzero constant
$-1$. This implies that $E/K$ is unramified at all prime divisors of
$2$. On the other hand, the conditions of Lemma \ref{onlyOneii} hold for all maximal ideals
$\b$ of $\Oc$ with {\sl odd} residual characteristic.  Now Remark \ref{trivialK} and Lemma \ref{onlyOneii}(c) imply that the Galois extension $K(\RR_u)/K(\sqrt{\Delta(u)})$ is unramified at every $\b$ of odd residual characteristic. This proves (a).

In order to prove (b), notice that
the condition (iii) implies that either all archimedean places
of both $E$ and $K$ are complex or all archimedean places of both
$E$ and $K$ are real. This implies that $E/K$ is unramified at all
infinite primes. Since the class number of $K$ is odd, the classical
results about Hilbert class fields \cite[Ch. 2, Sect. 1.2]{Koch}
imply that there is a maximal ideal $\b \subset \Oc$  such that
$E/K=K(\sqrt{D})/K$ is {\sl ramified} at $\b$. Since $E/K$ is
unramified at all prime divisors of $2$, the residual characteristic
of $\b$ is {\sl odd}, i.e., $2\not\in\b$.
This implies that
$$\Delta(u) =D \in \b.$$
Since $D=2^{2M}\cdot D_0$ and $\b$ is a prime (actually, maximal) ideal in $\Oc$, we have $D_0 \in \b$.
It also follows that
$$u(x) \bmod \b \in k(\b)[x]$$
has a multiple root. Now we are in a position to apply Lemma
\ref{onlyOneii}.  Since $K(\RR_u)\supset E$, the field extension
$K(\RR_u)/K$ is {\sl ramified} at $\b$. Applying Lemma
\ref{onlyOneii}, we conclude that $u(x) \bmod \b$ has exactly one
multiple root, this root is double and lies in $k(\b)$. In addition,
$$\Gal(K(\RR_u)/K) \subset \Perm(\RR_u)$$
contains a transposition. This implies that if $\Gal(K(\RR_u)/K)$ is
doubly transitive then $\Gal(K(\RR_u)/K)$ coincides with
$\Perm(\RR_u)\cong \ST_n$. Of course, this implies that $\Gal(K(\RR_u)/K(\sqrt{\Delta(u)}))=\A_n$.
\end{proof}

\begin{sect}[{\bf Generalized Mori quadruples}]
\label{MoriK}

  Let us consider a quadruple
$(g, \p, \bb,\cc)$  where $g$ is a positive integer, $\p$ is a
maximal ideal in $\Oc$  while $\bb$ and $\cc$ are elements of $\Oc$
that enjoy the following properties.

\begin{itemize}
\item[(i)]
The residue field $k(\p)=\Oc/\p$ is a finite field of {\sl odd}
characteristic. If $q$ is the cardinality of $k(\p)$ then every
prime divisor of $g$ is also a divisor of $\frac{q-1}{2}$. In
particular, if $g$ is even then $(q-1)$ is divisible by $4$.
\item[(ii)]
The residue  $\bb \bmod \p$ is a primitive element of $k(\p)$, i.e.,
it has multiplicative order $q-1$. In particular,
$$\bb \Oc +\p=\Oc.$$

The conditions (i) and (ii) imply that for each prime divisor $d$ of
$g$ the residue $\bb \bmod \p$ is {\sl not} a $d$th power in
$k(\p)$. Since $(q-1)$ is even, $\bb \bmod \p$ is {\sl not} a square
in $k(\p)$. So, if $d$ is a prime divisor of $2g$ then $\bb \bmod
\p$ is {\sl not} a $d$th power in $k(\p)$. If $2g$ is divisible by
$4$ then $g$ is even and $(q-1)$ is divisible by $4$, i.e., $-1$ is
a square in $k(\p)$. It follows  that $-4 \bb \bmod \p$ is {\sl not}
a square in $k(\p)$. Thanks to Theorem 9.1 of \cite[Ch. VI, Sect.
9]{Lang}, the last two assertions imply that the polynomial
$$x^{2g}-\bb \bmod \p \in k(\p)[x]$$
is irreducible over $k(\p)$. This implies that its Galois group over (the finite field)
$k(\p)$ is an order $2g$ cyclic group.
\item[(iii)]
$\cc \in \p, \ \cc-1 \in 2 \Oc$ and $$\Oc= \ \bb\Oc + \cc\Oc=
\bb\Oc+(2g+1)\Oc= \ 2g \Oc+\cc \Oc.$$
\end{itemize}

We call such a quadruple a {\sl generalized Mori quadruple} (in $K$).

\begin{ex}
Suppose that $K$ and $g$ are given. By Dirichlet's Theorem about primes in arithmetic progressions, there is a prime $p$ that does {\sl not} divide $(2g+1)$ and is congruent to $1$ modulo $2g$ . (In fact, there are infinitely many such primes.) Clearly, $p$ is {\sl odd}. Let us choose a maximal ideal $\p$ of $\Oc$ that contains $p$ and denote by $q$ the cardinality of the finite residue field $k(\p)$. Then $\fchar(k(\p))=p$ and  $q$ is a power of $p$. This implies that $q-1$ is divisible by $p-1$ and therefore is divisible by $2g$. Let us choose a generator  $\tilde{\bb}\in k(\p)$ of the multiplicative cyclic group $k(\p)^{*}$.  Let $r$ be a nonzero integer that is {\sl relatively prime} to $(2g+1)$. (E.g., $r=\pm 1, \pm 2$.)
Using Chinese Remainder Theorem, one may find $\bb \in \Oc$ such that
$$\bb \bmod \p=\tilde{\bb}, \ \bb-r \in (2g+1)\Oc.$$
(Clearly, $\bb \not\in\p$.)
 Now the same Theorem allows us to find  $\cc\in \p\subset \Oc$ such that $\cc-1 \in 2g\bb\Oc$.
Then $(g, \p, \bb,\cc)$ is a generalized Mori quadruple in $K$.
\end{ex}

Let us consider the polynomials
$$F(x)=F_{g, \p, \bb,\cc}(x)=x^{2g+1}-\bb x - \frac{\cc}{4} \in \Oc\left[\frac{1}{2}\right][x]\subset K[x]$$
and
$$U(x)=2^{2g+1} F(x/2)=x^{2g+1}-2^{2g}\bb x-2^{2g-1}\cc \in
\Oc[x]\subset K[x].$$
\end{sect}

\begin{thm}
\label{MoriG} Let $(g, \p, \bb,\cc)$  be a generalized Mori
quadruple in $K$. Assume also that there exists a maximal ideal
$\b_2\subset \Oc$ of residual characteristic $2$ such that the
ramification index $e(\b_2)$ of $\b_2$  (over $2$) in $K/\Q$ is
relatively prime to $(2g+1)$. Then:

\begin{itemize}
\item[(i)]
The polynomial
$$F(x)=F_{g, \p, \bb,\cc}(x)\in K[x]$$
 is irreducible over $K_{\b_2}$ and therefore
over $K$. In addition, the Galois extension $K(\RR_{F})/K$ is ramified at $\b_2$.
\item[(ii)]
The transitive Galois group
$$\Gal(F/\K)=\Gal(K(\RR_{F})/K)\subset \Perm(\RR_{F})=\ST_{2g+1}$$
 contains a cycle of
length $2g$. In particular, $\Gal(F/\K)$ is doubly transitive and is
not contained in $\A_{2g+1}$, and $\Delta(F)$ is not a square in
$K$.
\item[(iii)]
Assume that $K$ is a totally imaginary number field with odd class
number.  Then $\Gal(F/K)=\Perm(\RR_{F})$. If, in addition, $g>1$
then  $\End(J(C_F))=\Z$.
\item[(iv)]
Assume that $K$ is a totally imaginary number field with odd class
number and $g>1$. Then:

\begin{enumerate}
\item[(1)]
 for all primes $\ell$ the image
$\rho_{\ell,F}(\Gal(K))$ is an open subgroup of finite index in
$\Gp(T_{\ell}(J(C_F)),e_{\ell})$.
\item[(2)]
Let $L$ be a number field  that contains $K$ and $\Gal(L)$ be the
absolute Galois group of $L$, which we view as an open subgroup of
finite index in $\Gal(L)$. Then for all but finitely many primes
$\ell$ the image $\rho_{\ell,F}(\Gal(L))$ coincides with
$\Gp(T_{\ell}(J(C_F)),e_{\ell})$.
\end{enumerate}
\end{itemize}
\end{thm}

\begin{rem}
\label{twoK}
 If $K$ is a quadratic field then for every maximal ideal
$\b_2 \subset \Oc$ (with residual characteristic $2$) the ramification index $e(\b_2)$ of $\b_2$ in $K/\Q$ is
either $1$ or $2$: in both cases it is relatively prime to odd
$(2g+1)$. This implies that if $K$ is an {\sl imaginary quadratic
field} with {\sl odd class number} then all the conclusions of Theorem
\ref{MoriG} hold for every generalized Mori quadruple $(g, \p,
\bb,\cc)$. In particular,  the Galois extension $K(\RR_F)/K$ is {\sl ramified} at {\sl every} $\b_2$.

One may find the list of imaginary quadratic fields with {\sl small} ($\le 23$) odd class number  in \cite[pp. 322--324]{Arno};
see also \cite[Table 4 on p. 936]{Watkins}.
\end{rem}

\begin{proof}[Proof of Theorem \ref{MoriG}]
The $\b_2$-adic Newton polygon of  $F(x)$ consists of one {\sl
segment} that connects the points $(0,-2e(\b_2))$ and $(2g+1,0)$,
which are its only integer points, because $e(\b_2)$ and $(2g+1)$
are relatively prime and therefore $2e(\b_2)$ and $(2g+1)$ are
relatively prime. Now the irreducibility of $F(x)$ over $K_{\b_2}$
follows from Eisenstein--Dumas Criterion \cite[Cor. 3.6 on p.
316]{Mott}, \cite[p. 502]{Gao}. This proves (i). It also proves that the Galois extension $K(\RR_F)/K$ is {\sl ramified} at $\b_2$.

In order to prove (ii), let us consider the reduction
$$\tilde{F}(x)=F(x)\bmod \p
\Oc\left[\frac{1}{2}\right]=x^{2g+1}-\tilde{\bb} x \in k(\p)[x]$$
where
$$\tilde{\bb}=\bb \bmod \p \in k(\p).$$
So,
$$\tilde{F}(x)=x(x^{2g}-\tilde{\bb}) \in k(\p)[x].$$
We have already seen (Sect. \ref{MoriK}) that $x^{2g}-\tilde{\bb}$
is irreducible over $k(\p)$ and its Galois group is an order $2g$
cyclic group. We also know that $\tilde{\bb} \ne 0$ and therefore
the polynomials $x$ and $x^{2g}-\tilde{\bb}$ are relatively prime.
This implies that $K(\RR_F)/K$ is unramified at $\p$ and a
corresponding {\sl Frobenius element} in $\Gal(K(\RR_F)/K)\subset
\Perm(\RR_F)$ is a cycle of length $2g$. This proves (ii).  (Compare
with arguments on p. 107 of \cite{Mori}.)

The map $\alpha \mapsto 2\alpha$ is a $\Gal(K)$-equivariant
bijection between the sets of roots $\RR_F$ and $\RR_U$, which
induces a group isomorphism between permutation groups
$\Gal(\RR_F)\subset \Perm(\RR_F)$ and $\Gal(\RR_U)\subset
\Perm(\RR_U)$. In particular, the double transitivity of
$\Gal(\RR_F)$ implies the double transitivity of $\Gal(\RR_U)$. On
the other hand,
$$\Delta(U)=2^{(2g+1)2g}\Delta(F)=\left[2^{(2g+1)g}\right]^2\Delta(F).$$
This implies that $\Delta(U)$ is {\sl not} a square in $K$ as well.
The discriminant $\Delta(U)$ is given by the formula (Remark
\ref{discr})
$$D: =\Delta(U)=(-1)^{(2g+1)2g/2} (2g+1)^{2g+1} [-2^{2g-1}\cc]^{2g}+ (-1)^{2g(2g-1)/2} (2g)^{2g}
[-2^{2g}\b]^{2g+1}$$
$$=(-1)^{g}2^{2g(2g-1)}\left[(2g+1)^{2g+1}\cc^{2g}-2^{6g}g^{2g}\b^{2g+1}\right]=$$
$$2^{2[g(2g-1)]}\left\{(-1)^{g}\left[(2g+1)^{2g+1}\cc^{2g}-2^{6g}g^{2g}\b^{2g+1}\right]\right\}.$$
We have
$$D=2^{2M} D_0$$
where $M=g(2g-1)$ is a positive integer and
$$D_0=(-1)^{g}\left[(2g+1)^{2g+1}\cc^{2g}-2^{6g}g^{2g}\b^{2g+1}\right] \in
\Oc.$$ Since $\cc-1 \in 2\Oc$, we have $\cc^2-1 \in 4\Oc$ and
$$D_0 \equiv (-1)^{g} (2g+1)^{2g+1} \bmod 4\Oc.$$
Since $(2g+1)^{2g}=\left[(2g+1)^2\right]^g \equiv 1 \bmod 4$, we
conclude that
$$D_0 \equiv (-1)^{g} (2g+1) \bmod 4 \Oc.$$
This implies that
$$D_0-1 \in 4\Oc.$$
Applying Lemma \ref{ODDclass} to
$$n=2g+1, B=-2^{2g}\bb, C= -2^{2g-1}\cc, u(x)=U(x), M=g(2g-1), N=2g,$$
 we conclude that
doubly transitive $\Gal(U/K)$ coincides with $\Perm(\RR_U)$ and
therefore $\Gal(F/K)$ coincides with $\Perm(\RR_F)\cong \ST_{2g+1}$.
If $g>1$ then Theorem \ref{endo} tells us that $\End(J(C_F))=\Z$.
This proves (iii). We also obtain that there exists a maximal ideal
$\b \subset \Oc$ with odd residual characteristic such that
$$U(x) \bmod \b \in k(\b)[x]$$
has exactly one multiple root, this root is double and lies in
$k(\b)$.  Since $$F(x)=\frac{U(2x)}{2^{2g+1}},$$ we obtain that
$$F(x)\bmod \b \Oc\left[\frac{1}{2}\right] =\frac{U(2x)}{2^{2g+1}}\bmod \b \in k(\b)[x].$$
This implies that the polynomial $F(x)\bmod \b
\Oc\left[\frac{1}{2}\right] \in k(\b)[x]$ has exactly one multiple
root, this root is double and lies in $k(\b)$.
The properties of $F(x) \bmod \b \Oc\left[\frac{1}{2}\right]$ imply
that $J(C_F)$ has a {\sl semistable reduction} at $\b$ with {\sl
toric dimension} $1$. Now it follows from \cite[Th. 4.3]{ZarhinTAMS}
that for for all primes $\ell$ the image $\rho_{\ell,F}(\Gal(K))$ is
an open subgroup of finite index in
$\Gp(T_{\ell}(J(C_F)),e_{\ell})$. It follows from \cite[Th. 1]{Hall}
that if $L$ is a number field containing $K$ then for all but
finitely many primes $\ell$ the image $\rho_{\ell,F}(\Gal(L))$
coincides with $\Gp(T_{\ell}(J(C_F)),e_{\ell})$.  This proves (iv).
\end{proof}

\begin{cor}
We keep the notation of Theorem \ref{MoriG}.
Let $K$ be an imaginary quadratic field with odd class
number.  Let $(g, \p, \bb,\cc)$  be a generalized Mori
quadruple in $K$ and
$$F(x)=F_{g, \p, \bb,\cc}(x) \in K[x].$$
Then
$$\Gal(K(\RR_F)/K(\sqrt{\Delta(F)}))=\A_{2g+1}$$
and
the Galois extension $K(\RR_F)/K(\sqrt{\Delta(F)})$ is unramified everywhere   outside $2$ and ramified at all prime divisors of $2$.
\end{cor}

\begin{proof}
As above, let us consider the polynomial
$$U(x)=2^{2g+1} F(x/2)=x^{2g+1}-2^{2g}\bb x-2^{2g-1}\cc \in
\Oc[x]\subset K[x].$$
We have
$$K(\RR_F)=K(\RR_U), \ K(\sqrt{\Delta(F)})=K(\sqrt{\Delta(U)}).$$
Since
$$\ST_{2g+1}=\Perm(\RR_U)=\Gal(U/K)=\Gal(K(\RR_U)/K),$$
we have
$$\Gal(K(\RR_U)/K(\sqrt{\Delta(U)}))=\A_{2g+1}.$$
It follows from Remark \ref{twoK} that the Galois extension $K(\RR_U)/K$ is {\sl ramified} at every prime divisor of $2$ (in $K$). On the other hand, Lemma \ref{ODDclass}(a) (applied to $u(x)=U(x)$) tells us that the quadratic extension $K(\sqrt{\Delta(U)})/K$ is {\sl unramified} at every prime divisor of $2$ (in $K$). Since all the field extensions involved are Galois,
$K(\RR_U)/K(\sqrt{\Delta(U)})$ is {\sl ramified} at every prime divisor of $2$ (in $K(\sqrt{\Delta(U)})$).

Since $K$ is purely {imaginary}, $K(\sqrt{\Delta(U)})$ is also purely imaginary and therefore (its every field extension, including) $K(\RR_U)$ is unramified at all infinite places (in $K(\sqrt{\Delta(U)})$).

Remark \ref{trivialK} and Lemma \ref{ODDclass}(a) (applied to $u(x)=U(x)$) imply that the field extension  $K(\RR_U)/K(\sqrt{\Delta(U)})$ is {\sl unramified} at all maximal ideals $\b$ in $\Oc$ with odd residual characteristic.
\end{proof}

\section{Corrigendum
to \cite{ZarhinCEJM}}

\begin{itemize}
\item
Page 660, the 6th displayed formula: insert $\subset$ between
$\End_{\Gal(K)}V_{\ell}(X)$ and $\End_{\Q_{\ell}}V_{\ell}(X)$.
\item
Page 662, Theorem 2.6, line 3: $r_1$ should be $r_2$.
\item
Page 664, Remark 2.16:  The reference to [23, Theorem 1.5] should be
replaced by [23, Theorem 1].

\item
Page 664, Theorem 2.20.
The following additional condition on
$\ell$ was inadvertently omitted.

\noindent `` (iii) {\sl If $C$ is the center of $\End(X)$ then
$C/\ell C$ is the center of} $\End(X) / \ell  \End(X)$."

In addition,  ``be" on the last line should be ``is".
\item
Page 666, Theorem. 3.3. Line 2: $\ell$ should be assumed to be in
$P$, i.e.  one should read ``{\sl Then for all but finitely many}
$\ell \in P$ . . .". In addition,   $X_n$ should be $X_{\ell}$
throughout lines 3--6.
\item
Page 668, Lemma 3.9, line 1: $\Isog_P$ should be $\Is_P$.
\item
Page 668, Theorem 3.10, line 1: replace $\Isog_P((X \times X^t)^8,
K, 1)$ by $\Is_P((X \times X^t)^4, K, 1)$.
\item
Page 670, Sect. 5.1, the first displayed formula: $t$ should be $g$.
\item
Page 672, line 9: $X'_{\ell}$ should be $X_{\ell}$.
\end{itemize}

(I am grateful to Kestutis Cesnavicius, who has sent me this list of
typos.)

\end{document}